\documentclass[11pt,amssymb]{amsart}
\usepackage{amsmath,amssymb}
\usepackage[mathscr]{eucal}
\usepackage[all,cmtip]{xy}

\newcommand{\Z}{{\mathbb Z}}

\newcommand{\Tr}{\mathrm{Tr}}

\newcommand{\Ga}{\mathrm{Gal}}

\newtheorem{theorem}{Theorem}[section]
\newtheorem{proposition}[theorem]{Proposition}
\newtheorem{lemma}[theorem]{Lemma}

.240pk scaled 1200 .240pk
\DeclareFontFamily{U}{wncy}{}
\DeclareFontShape{U}{wncy}{m}{n}{<->wncyr10}{}
\DeclareSymbolFont{mcy}{U}{wncy}{m}{n}
\DeclareMathSymbol{\Sha}{\mathord}{mcy}{"58}

\usepackage{hyperref}
\usepackage{xcolor}
\usepackage[shortlabels]{enumitem}
\usepackage{cancel}
\usepackage{tikz-cd}
\usepackage{comment}

\newcommand\extrafootertext[1]{%
    \bgroup
    \renewcommand\thefootnote{\fnsymbol{footnote}}%
    \renewcommand\thempfootnote{\fnsymbol{mpfootnote}}%
    \footnotetext[0]{#1}%
    \egroup
}

\newcommand{\inv}{^{-1}}
\newcommand{\noi}{\noindent}
\newcommand{\ov}{\overline}
\newcommand{\wt}{\widetilde}

\newcommand{\Csf}{\mathsf{C}}

\newcommand{\calH}{\mathcal{H}}
\newcommand{\G}{\mathbb{G}}
\newcommand{\sig}{\sigma}
\newcommand{\eps}{\epsilon}
\newcommand{\iso}{\cong}
\newcommand{\semi}{\rtimes}
\newcommand{\frakm}{\mathfrak{m}}

\newcommand{\lp}{\left(}
\newcommand{\rp}{\right)}
\newcommand{\lb}{\left\{}
\newcommand{\rb}{\right\}}
\newcommand{\into}{\hookrightarrow}
\DeclareMathOperator{\SU}{SU}
\DeclareMathOperator{\M}{M}

\DeclareMathOperator{\GL}{GL}

\DeclareMathOperator{\Id}{Id}
\DeclareMathOperator{\characteristic}{char}

\DeclareMathOperator{\Gal}{Gal}

\theoremstyle{definition}
\newtheorem{definition}[theorem]{Definition}
\newtheorem{remark}[theorem]{Remark}

\renewcommand{\tilde}{\widetilde}

\begin{document}

\title[Unitary groups]{On abstract homomorphisms of some special unitary groups}

\author[I.~Rapinchuk]{Igor A. Rapinchuk}
\author[J.~Ruiter]{Joshua Ruiter}

\address{Department of Mathematics, Michigan State University, East Lansing, MI 48824, USA}

\email{rapinchu@msu.edu}

\address{Department of Mathematics, Michigan State University, East Lansing, MI
48824, USA}

\email{ruiterj2@msu.edu}

\maketitle

\begin{abstract}
We analyze the abstract representations of the groups of rational points of even-dimensional quasi-split special unitary groups associated with quadratic field extensions. We show that, under certain assumptions, such representations have a standard description, as predicted by a conjecture of Borel and Tits \cite{BT}. Our method extends the approach introduced by the first author in \cite{IR} to study abstract representations of Chevalley groups and is based on the construction and analysis of a certain algebraic ring associated to a given abstract representation.
%We extend methods introduced by the first author in \cite{IR} to analyze representations of the group $G(k)$ of rational points of an even-dimensional quasi-split special unitary group associated with a quadratic field extension $L/k$ (in characteristic zero). Our main result is that an abstract representation $\rho:G(k) \to \GL_m(K)$ (with $K$ algebraically closed, characteristic zero) admits a standard description, i.e. a factorization $\rho = \sig \circ F$ where $\sig:G(A) \to \GL_m(K)$ is a morphism of algebraic groups and $F:G(k) \to G(A)$ arises from a ring homomorphism $f:k \to A$. A key step is the construction of the algebraic ring $A$ which extends a similar construction in \cite{IR}. In particular, this verifies a special case of a conjecture of Borel and Tits \cite[8.19]{BT}.
\end{abstract}

\section{Introduction}\label{S-Introduction}

\extrafootertext{
2010 \textit{Mathematics subject classification}: primary 20G15, secondary 20G35.

\textit{Key words and phrases.} Abstract homomorphism, special unitary group.
}

The goal of this paper is to analyze the abstract representations of the groups of rational points of even-dimensional quasi-split special unitary groups associated with quadratic field extensions. Our main result, whose precise formulation is given in Theorem \ref{T-MainTheorem} below, is that, under certain assumptions, such representations have the expected description, as predicted by the following longstanding conjecture of Borel and Tits (\cite[8.19]{BT}).

As usual, for
%In this paper, we extend the scope of the methods introduced in \cite{IR} and \cite{IR1} in order to analyze the abstract representations of the groups of rational points of even-dimensional quasi-split special unitary groups associated with quadratic field extensions; in particular, we verify a conjecture of Borel and Tits (\cite[8.19]{BT}) for abstract representations of such groups. Before giving the precise formulation of our main result, we begin by recalling the statement of the Borel-Tits conjecture.
an algebraic group $G$ defined over a field $k$, we denote by $G^+$ the (normal) subgroup of $G(k)$ generated by the
%$k$-points of split (smooth) connected unipotent $k$-subgroups.
the $k$-points of the unipotent radicals of the $k$-defined parabolic subgroups of $G$. With these notations, Borel and Tits conjectured the following.
\vskip3mm

\noindent (BT) \ \ \parbox{14.5cm}{Let $G$ and $G'$ be algebraic groups defined over infinite fields $k$ and $k'$, respectively. If $\rho \colon G(k) \to G'(k')$ is any abstract homomorphism such that $\rho(G^+)$ is Zariski-dense in $G'(k'),$ then {there exists a commutative finite-dimensional $k'$-algebra $A$ and a ring homomorphism $f_A : k \to A$ such that  $\rho = \sigma \circ r_{A/k'} \circ F$, where $F : G(k) \to G_A(A)$ is induced by $f_A$ ($G_A$ is the group obtained by change of scalars), $r_{A/k'} : G_A(A) \to R_{A/k'}(G_A)(k')$ is the canonical isomorphism  (here $R_{A/k'}$ denotes the functor of restriction of scalars), and $\sigma$ is a rational $k'$-morphism of $R_{A/k'} (G_A)$ to $G'.$}}

\vskip3mm

If an abstract homomorphism $\rho : G(k) \to G'(k')$ admits a factorization as in (BT), we will say that $\rho$
has a {\it standard description}.

\begin{remark}
It has been pointed out by B.~Conrad and G.~Prasad that by using the results of \cite[Chapter 9]{CGP1} (particularly Proposition 9.9.2), 
%using \cite[Proposition 9.9.2]{CGP1},
%, using techniques from the theory of pseudo-reductive groups (developed in \cite{CGP}, Chapter 9), 
one can construct counterexamples to (BT) over all local and global function fields of characteristic 2 (or, more generally, over any field $k$ of characteristic 2 such that $[k:k^2] = 2$). The groups that arise in these counterexamples are perfect and $k$-simple. So, one should exclude fields of characteristic 2 in
%even though this case was not excluded by Borel and Tits, one should require that char $k \neq 2$ in
the statement of (BT).
\end{remark}

In \cite{IR}, the first author introduced a method, based on the construction and analysis of certain algebraic rings, for studying abstract representations of the elementary subgroups of simply-connected Chevalley groups over commutative rings. Using these techniques, he obtained a general result on abstract representations that, in particular, yielded (BT)
%One consequence of the main results of \cite{IR} was a proof of (BT)
in the case where $k$ is a field of characteristic $\neq 2$ or 3, $k' = K$ is an algebraically closed field of characteristic 0, and $G$ is a split simply-connected $k$-group. Subsequently, this approach was extended in \cite{IR1} to confirm (BT) for abstract representations of groups of the form $\mathbf{SL}_{n, D}$, where $D$ is a finite-dimensional central division algebra over a field of characteristic zero. In positive characteristic, results on this problem were obtained by Seitz in \cite{Seitz}, which were subsequently revisited by Boyarchenko and the first author in \cite{BR} from the point of view of algebraic rings. Furthermore, the first author recently employed the methods of \cite{IR}, together
with some considerations on central extensions, to obtain the first unconditional rigidity statements for finitely generated linear groups other than arithmetic groups/lattices (such as $\mathrm{SL}_n(\Z[x])$, for $n \geq 3$) --- cf. \cite{IR-Curves}. We refer the reader to \cite{IR3} for a more extensive overview of work on (BT) and its connections to various classical forms of rigidity. 

%{\textcolor{red}{Some additional discussion of known results here? In particular, addressing $\characteristic = 2,3$ cases and how the conjecture likely fails in those cases. Maybe mention section 7.2 of \cite{CGP1}, as in the footnote on page 1 of \cite{IR1}. Maybe also mention of Seitz result for positive characteristic reductive groups. Maybe also mention of \url{https://arxiv.org/pdf/1603.03271.pdf}}}

In the present paper, we consider the following situation. Let $k$ be a field of characteristic 0 and let $L = k(\sqrt{d})$ be a quadratic extension. We set $\tau : L \to L$ to be the nontrivial element of $\Gal(L/k)$ and note that for any commutative $k$-algebra $R$, the action of $\tau$ naturally extends to $R_L : = R \otimes_k L$ via the second factor. For ease of notation, we will write $\tau(x) = \ov x$ for $x \in R_L.$ Next, fix an integer $n \geq 2$ and let $V = L^{2n}$ be a $2n$-dimensional $L$-vector space equipped with a (skew-)hermitian form $h \colon V \times V \to L$. We will assume that $h$ has maximal Witt index, so that, with respect to a suitable basis
%$\{e_1, e_2, \dots, e_{2n-1}, e_{2n} \}$
of $V$, the matrix of $h$ is
\[
H =
	\begin{pmatrix}
		0 & -1   \\
		1 & 0  \\
		& & \ddots   \\
		& & & 0 & -1 \\
		& & & 1 & 0
	\end{pmatrix}
\]
Let $G = \mathrm{SU}_{2n}(L,h)$ be the corresponding special unitary group. Explicitly, for a commutative $k$-algebra $R$, we have
\[
	G(R) = \{X \in \mathrm{SL}_{2n}(R_L) \mid X^* H X = H \},
\]
where for $X = (a_{ij})$, we let $X^* = (\ov a_{ji})$ denote the conjugate transpose matrix. It is well-known that $G$ is a quasi-split simply-connected $k$-group (in fact, a $k$-form of $\mathrm{SL}_{2n}$) whose relative root system $\Phi_k$ is of type $\mathsf{C}_n$ (see, for example, \cite[Ch. V, \S23.8]{BorelAG} or \cite[Ch. 2, \S2.3.3]{Pl-R}).  With these notations, our main result is as follows.

\begin{theorem}
\label{T-MainTheorem}
Let $L = k(\sqrt{d})$ be a quadratic extension of a field $k$ of characteristic 0, and for $n \geq 2$, set $G = \mathrm{SU}_{2n}(L, h)$ to be the special unitary group of a (skew-)hermitian form $h : L^{2n} \times L^{2n} \to L$ of maximal Witt index. Let $K$ be an algebraically closed field of characteristic 0 and consider an abstract representation
\[
	\rho : G(k) \to \mathrm{GL}_m(K).
\]
Set $H = \overline{\rho(G(k))}$ to be the Zariski closure of the image of $\rho$ (we note that $H$ is automatically connected --- see Lemma \ref{H is connected}). Then if the unipotent radical $U = R_u(H)$ of $H$ is commutative, there exists a commutative finite-dimensional $K$-algebra $A$, a ring homomorphism $f : k \to A$ with Zariski-dense image, and a morphism of algebraic $K$-groups $\sigma : G(A) \to H$ such that $\rho = \sigma \circ F$, where $F : G(k) \to G(A)$ is the group homomorphism induced by $f$.
\end{theorem}

\noindent (In this statement, we view $G(A)$ as an algebraic $K$-group using the functor of restriction of scalars --- see \S\ref{S-Rationality} for further details.)

\vskip2mm

The proof of this result proceeds along the following lines. First, since $G$ is a simply-connected $k$-group whose relative root system $\Phi_k$ is of type $\mathsf{C}_n$, it follows that $G$ contains a $k$-split simply-connected $k$-group $G_0 = \mathrm{Sp}_{2n}$ of type $\mathsf{C}_n$ (see \cite[Th\'eor\`eme 7.2]{BT1} or \cite[Theorem C.2.30]{CGP1}). We consider the restriction of $\rho$ to $G_0(k)$, and use the construction given in \cite{IR} to associate to $\rho|_{G_0(k)}$ an algebraic ring $A$, together with a ring homomorphism $f : k \to A$ with Zariski-dense image.
Since $k$ and $K$ are both fields of characteristic 0, $A$ is in fact a finite-dimensional $K$-algebra (by \cite[Lemma 2.13(ii), Proposition 2.14]{IR}). As shown in \cite{IR}, the algebraic ring $A$ plays a central role in proving that $\rho|_{G_0(k)}$ has a standard description; it turns out that $A$ also suffices for the analysis of
%the given representation
$\rho$. More precisely, following the general strategy of \cite{IR} and \cite{IR1}, we first show that $\rho$ lifts to a representation $\wt{\sigma}: \wt G(A) \to \mathrm{GL}_m(K)$, where $\wt G(A)$ is the generalized Steinberg group introduced by Stavrova \cite{Stavrova} (which builds on an earlier construction due to Deodhar \cite{Deodhar}). Then, using the fact that the kernel of the canonical map $\wt G(A) \to G(A)$ is central (which extends a result of Stavrova to the present situation), together with our assumption that the unipotent radical $R_u(H)$ is commutative, we establish the existence of the required algebraic representation $\sigma: G(A) \to \mathrm{GL}_m(K).$

The structure of the paper is as follows. We begin by recalling in \S\ref{S-Elementary} several key definitions and statements pertaining to elementary subgroups of isotropic reductive group schemes (defined by Petrov and Stavrova \cite{PetrovStavrova}) that are needed for our purposes. We also discuss some relevant aspects of Stavrova's generalization in \cite{Stavrova} of the classical theory of Steinberg groups and central extensions. Next, in \S\ref{S-AlgebraicRing}, we introduce the algebraic ring $A$ and, after establishing some preliminary statements, show that $\rho$ lifts to a representation $\wt{\sigma}$ of $\wt{G}(A).$ Finally, in \S\ref{S-Rationality}, we follow the general strategy of \cite{IR} to prove that $\wt{\sigma}$ descends to a rational representation $\sigma$ of $G(A).$

\vskip2mm

\noindent {\sc Notations and conventions:} All rings will be assumed to be unital and commutative. Unless stated otherwise, we will denote by $G$ the special unitary group $\mathrm{SU}_{2n}(L,h)$ introduced above.
%A ring is indecomposable if it contains no nontrivial idempotents, and an algebraic ring is connected if it is connected as a variety.

\section{Elementary subgroups and Steinberg groups}\label{S-Elementary}

In this section, we recall several aspects of the theory of elementary subgroups of isotropic reductive group schemes, developed by Petrov and Stavrova \cite{PetrovStavrova}, as well as relevant points of Stavrova's \cite{Stavrova} generalization of the classical Steinberg group. In view of our later applications, we will focus on the case of $\SU_{2n}(L,h)$ ($n \ge 2$).
% which has relative root system $\Phi_k$ of type $\mathsf{C}_n$ (in particular, it is reduced and irreducible).

%In view of our applications of this material later in subsequent sections, we will focus mainly on the case where $G = \mathrm{SU}_{2n}(L,h)$ ($n \geq 2$) as in \S\ref{S-Introduction}. Thus, the corresponding relative root system $\Phi_k$ is of type $\mathsf{C}_n$ (in particular, it is reduced and irreducible).

\subsection{Elementary subgroups}
Suppose $\mathcal{G}$ is a reductive group scheme over a ring $R$ that is isotropic of rank $\geq 1$ (i.e. every semisimple normal $R$-subgroup of $\mathcal{G}$ contains a 1-dimensional split $R$-torus). Then $\mathcal{G}$ contains a pair of opposite parabolic $R$-subgroups $P$ and $P^-$ that intersect properly every semisimple normal $R$-subgroup of $\mathcal{G}$ (such subgroups are sometimes referred to as \textit{strictly parabolic}). In \cite{PetrovStavrova}, the corresponding elementary subgroup $E_P(R)$ is then defined as the subgroup of $\mathcal{G}(R)$ generated by $U_{P}(R)$ and $U_{P-}(R)$, where $U_{P}$ and $U_{P^-}$ are the unipotent radicals of $P$ and $P^-$, respectively (we note that when $R = k$ is a field, then $E_P(k)$ coincides with the group $\mathcal{G}^+$ appearing in the statement of (BT)). The main result of \cite{PetrovStavrova} (see also \cite[Theorem 2.4]{StavrovaK1}) is that if for any maximal ideal $\frakm \subset R$, the group $\mathcal{G}_{R_\frakm}$ is isotropic of rank $\ge 2$, then $E_P(R)$ does not depend on the choice of a strictly parabolic subgroup $P$. This assumption is automatically satisfied in all situations considered in this paper, so, to simplify notations, we will denote the elementary subgroup simply by $E(R)$.

Furthermore, in analogy with elementary subgroups of Chevalley groups, Petrov and Stavrova provide a description of $E(R)$ in terms of generators that satisfy certain generalized Chevalley commutator relations. The following statement collects the relevant parts of \cite[Theorem 2]{PetrovStavrova} and \cite[Lemma 2.14]{Stavrova} in the case of $G = \SU_{2n}(L,h)$ that will be needed for our analysis.

\begin{theorem}
\label{root subgroup maps}
Let $G = \SU_{2n}(L,h)$, fix a maximal $k$-split torus $S \subset G$, and denote by $\Phi_k$ the corresponding relative root system (of type $\Csf_n$), viewed as a subset of the character group $X^*(S)$. Then for every $\alpha \in \Phi_k$, there exists a vector $k$-group scheme $V_\alpha$ and a closed embedding of schemes
\[
	X_\alpha:V_\alpha \to G
\]
such that for any $k$-algebra $R$, we have the following:
\begin{enumerate}
	\item For any $v,w \in V_\alpha(R)$,
	\[
		X_\alpha(v) \cdot  X_\alpha(w) = X_{\alpha}(v+w)
	\]
	In particular, $X_\alpha(0) = 1$.
	\item For any $s \in S(R)$ and $v \in V_\alpha(R)$,
	\[
		s \cdot X_\alpha(v) \cdot s \inv = X_\alpha \Big( \alpha(s) v \Big)
	\]
	\item (Chevalley commutator formula) For any $\alpha,\beta \in \Phi_k$ such that $\alpha \neq \pm \beta$, and for all $u \in V_\alpha(R), v \in V_\beta(R)$,
	\[
		\Big[ X_\alpha(u), X_\beta(v) \Big] = \prod_{\substack{{i,j \ge 1} \\ i\alpha+j\beta \in \Phi_k}} X_{i\alpha+j\beta} \Big( N_{ij}^{\alpha \beta}(u,v) \Big)
	\]
	for some polynomial maps $N_{ij}^{\alpha \beta}:V_\alpha(R) \times V_\beta(R) \to V_{i\alpha+j\beta}(R)$. The map $N_{ij}^{\alpha \beta}$ is homogeneous of degree $i$ in the first variable and homogeneous of degree $j$ in the second variable.
	\item The elementary subgroup $E(R)$ is generated by the elements $X_\alpha(v)$ for all $\alpha \in \Phi_k$ and all $v \in V_\alpha(R)$.
\end{enumerate}
\end{theorem}

\begin{remark}
\label{G=E}
\ \

\vskip1mm

\noindent (a) \parbox[t]{16cm}{In view of the structure of the root system $\mathsf{C}_n$, it is easy to see that the products on the right hand  side of (3) contain at most two terms with the possible values of $i$ and $j$ lying in $\{1, 2 \}.$ Moreover, in the cases where there are two terms, these commute with each other.}

\vskip1mm

\noindent (b) \parbox[t]{16cm}{It follows from the definitions that the construction of the elementary subgroup is functorial in $R$ (i.e. a ring homomorphism $R_1 \to R_2$ gives rise to a group homomorphism $E(R_1) \to E(R_2)$), and that it is compatible with finite products (i.e. $E(R_1 \times \cdots \times R_n) = E(R_1) \times \cdots \times E(R_n)$). Furthermore, since $G = \SU_{2n}(L,h)$ is a quasi-split simply-connected $k$-group, this observation and \cite[Lemma 5.2]{Stavrova} imply that $G(R) = E(R)$ for any ring $R$ that is a finite product of local $k$-algebras.}
\end{remark}

Continuing with the notations of the theorem, the dimension of $V_\alpha$ is simply the dimension of the relative root space associated to $\alpha$. Thus, in our situation, there are two possibilities: if $\alpha$ is a long root, then $V_\alpha = \G_a$, so that $V_\alpha(R) = R$ for any $k$-algebra $R$; on the other hand, if $\alpha$ is short, then $V_\alpha \iso (\G_a)^2$, and we have $V_\alpha(R) = R \otimes_k L = R_L$ (to make the identification $(\G_a)^2(R) = R^2 \iso R_L$, we use the fact that $k^2 \iso L$ as a $k$-vector space).

%This dimension is also determined by the length of $\alpha$: for long roots $V_\alpha = \G_a$, and for short roots $V_\alpha = (\G_a)^2$. It is sometimes useful to view $V_\alpha(R)$ as a ring rather than just an additive group. For long roots, $V_\alpha(R) = \G_a(R) = R$ uses multiplication in $R$, and for short roots $V_\alpha(R) = \G_a(R)^2 = R^2 = R \otimes_k L$ has the usual multiplication as an $L$-algebra. So for all roots, $R$ sits inside $V_\alpha(R)$, either as the whole group, or as the subgroup $R \otimes_k 1 \subset R \otimes_k L$.

For the calculations that we will carry out in subsequent sections, it will be useful to make the statement of Theorem \ref{root subgroup maps} more explicit, as follows. Fix an algebraic closure $\bar{k}$ of $k$ and consider the maximal $k$-split torus $S$ of $G$ (we will identify $G$ with $G(\bar{k})$ for the sake of concreteness) consisting of diagonal matrices of the form
\[
\begin{pmatrix}
		t_1 &    \\
		 & t_1^{-1}  \\
		& & \ddots   \\
		& & & t_n &  \\
		& & &  & t_n^{-1}
	\end{pmatrix}
\]
with $t_i \in \bar{k}.$ For $1 \leq i \leq 2n$, let $\alpha_i : S \to \G_m$ be the character that maps such a matrix to its $i$-th diagonal entry. Then one easily checks that
\[
	\Phi_k = \lb \pm 2 \alpha_i : i \text{ odd} \rb \cup \lb \pm \alpha_i \pm \alpha_j : i \neq j, \text{ both odd} \rb
\]
with $1 \le i,j \le 2n$.
%$i, j \in \{1, 2, \cdots, 2n \}.$
Furthermore, the morphisms $X_{\alpha}$ look as follows. For a ring $R$, we denote by $E_{ij}(x) \in \M_{2n}(R)$ the matrix with $x$ in the $ij$-th entry and 0 in all other entries. Also, if $R$ is a $k$-algebra, then, as we observed previously, the action of the nontrivial element $\tau \in \Ga(L/k)$ extends to $R_L = R \otimes_k L$, and we will write $\tau(a) = \ov a$ for $a \in R_L.$ Then the root group morphisms for the long roots are
\begin{align*}
	X_{2\alpha_i}(R)&:R \to G(R) \qquad x \mapsto 1 + E_{i,i+1}(x) \\
	X_{-2\alpha_i}(R)&:R \to G(R) \qquad x \mapsto X_{2\alpha_i}(x)^t =  1 + E_{i+1,i}(x),
\end{align*}
and for the short roots, the morphisms are
\begin{align*}
	X_{\alpha_i - \alpha_j}(R)&:R_L \to G(R) \qquad x \mapsto 1 + E_{ij}(x) - E_{j+1,i+1}(\ov x) \\
	X_{-\alpha_i + \alpha_j}(R)&:R_L \to G(R) \qquad x \mapsto X_{\alpha_i - \alpha_j}(x)^t = 1 + E_{ji}(x) - E_{i+1,j+1}(\ov x), \\
X_{\alpha_i + \alpha_j}(R)&:R_L \to G(R) \qquad x \mapsto 1 + E_{i',j'+1}(x) + E_{j',i'+1}(\ov x), \\
	X_{-\alpha_i - \alpha_j}(R)&:R_L \to G(R) \qquad x \mapsto X_{\alpha_i + \alpha_j}(x)^t = 1 + E_{j'+1,i'}(x) + E_{i'+1,j'}(\ov x),
\end{align*}
where for a pair $i, j$, we set $i' = \min(i,j)$ and $j' = \max(i,j)$.

By direct calculation, one obtains explicit formulas for the polynomial maps $N_{ij}^{\alpha \beta}$ appearing in Theorem \ref{root subgroup maps}. To formulate the result, given a $k$-algebra $R$, we let
\begin{align*}
	\Tr: R_L \to R \qquad a \mapsto a + \bar{a}
\end{align*}
denote the extension of the usual trace map $\Tr_{L/k} : L \to k$ to $R_L.$ Also, for $\delta = \pm 1$ and $v \in R_L$, we set
\[
	v_\delta =
	\begin{cases}
		v & \delta = 1 \\
		\ov v & \delta = -1
	\end{cases}
\]
The following lemma describes the maps $N_{ij}^{\alpha \beta}$ arising in all nontrivial Chevalley commutator relations for $\SU_{2n}(L,h)$.

%%%%%%%%%%%%%%%%%%%%%%%%%%%%%%%%%%%%%%%%%%%%%%%%%%%%%%%%%%%%%%%%%%%%%%%
% This is a remark describing the structure of C_n, in paricular noting that all nontrivial Chevalley commutator relations arise from a copy of C_2 inside C_n.
%%%%%%%%%%%%%%%%%%%%%%%%%%%%%%%%%%%%%%%%%%%%%%%%%%%%%%%%%%%%%%%%%%%%%%%
%\noi For roots $\alpha, \beta \in \Phi_k$, following \cite{PetrovStavrova} we denote by $(\alpha, \beta)$ the set of linear combinations $i\alpha + j\beta$ which are roots.
%\begin{remark}
%\label{C2}
%All nontrivial Chevalley commutator relations in Theorem \ref{root subgroup maps} part (3) arise from a copy of $\Csf_2$ sitting inside $\Csf_n$. That is, if $\alpha,\beta \in \Phi_k$ ($\alpha \neq - \beta$) and the product on the right side is nontrivial (equivalently $\alpha+\beta \in \Phi_k$ or $(\alpha,\beta)$ is nonempty), then $\lb \alpha, \beta \rb \cup (\alpha,\beta)$ sits in a copy of $\Csf_2$.
%More explicitly, if $\alpha,\beta$ are short roots and $\alpha+\beta \in \Phi_k$, then $(\alpha, \beta) = \lb \alpha+\beta \rb$ and $\alpha + \beta$ is long. Conversely, if $\alpha+\beta$ is a long root, then $\alpha,\beta$ are both short. Similarly, if $\alpha$ is long and $\beta$ is short and $\alpha + \beta \in \Phi_k$, then $(\alpha,\beta) = \lb \alpha+\beta, \alpha+2\beta \rb$ where $\alpha+\beta$ is short and $\alpha+2\beta$ is long. Conversely, if $\alpha+\beta$ is a short root, then $\alpha,\beta$ have different lengths. The sum of two long roots is never a root, so this describes all possible situations in which the set $(\alpha,\beta)$ is nonempty.
%\end{remark}

\begin{lemma}
\label{N is surjective}
Let $\alpha, \beta \in \Phi_k$ be relative roots such that $\alpha+\beta \in \Phi_k$ and let $R$ be a $k$-algebra.
\begin{enumerate}
	\item Suppose $\alpha$ and $\beta$ are short and $\alpha+\beta$ is short. Then (relabelling if necessary) we can write $\alpha = \alpha_i - \alpha_j$ and $\beta = \alpha_j - \alpha_{\ell}$ for three distinct indices $i,j, \ell$, and
	\[
		N_{11}^{\alpha \beta}(u,v) = uv \qquad N_{11}^{\beta \alpha}(v,u) = -uv
	\]
	\item Suppose $\alpha$ and $\beta$ are both short and $\alpha+\beta$ is long. Then (relabelling if necessary) we can write $\alpha = \eps(\alpha_i - \alpha_j), \beta = \omega (\alpha_i + \alpha_j)$ for some $\eps = \pm 1, \omega = \pm 1$, with $i < j$, and
\begin{align*}
	N_{11}^{\alpha\beta}(u,v) = \omega  \Tr( u_{-\eps \omega} v ) \qquad N_{11}^{\beta \alpha}(v,u) = -\omega  \Tr( u_{-\eps \omega} v )
\end{align*}	
	for all $u \in V_\alpha(R), v \in V_\beta(R)$.
	\item Suppose $\alpha$ is short and $\beta$ long. Then we can write $\alpha = \eps \alpha_i + \omega \alpha_j$ and $\beta = - \eps 2 \alpha_i$ for some $\eps = \pm 1, \omega = \pm 1$ and $i \neq j$, and
\begin{align*}
	N_{11}^{\alpha\beta}(u,v) &= \omega u_{-c_{ij}} v \qquad \hspace{2pt} N_{11}^{\beta \alpha}(v,u) = - \omega u_{-c_{ij}} v \\
	N_{21}^{\alpha\beta}(u,v) &= -\eps \omega v u \ov u \qquad N_{12}^{\beta \alpha}(v,u) = \eps \omega vu \ov u
\end{align*}
	for all $u \in V_\alpha(R), v \in V_\beta(R)$, where
\[
c_{ij} =
	\begin{cases}
		1 & i < j \\
		-1 & i > j
	\end{cases}
\]
\end{enumerate}
In particular, whenever it is defined, the map $N_{11}^{\alpha \beta}$ is surjective.
\end{lemma}

%%%%%%%%%%%%%%
% I modified part (1) to reverse the roles of alpha, beta. This also reversed the roles of epsilon, omega.
%%%%%%%%%%%%%%

%%%%%%%%%%%
% Previous version of part (2) of the lemma. I reversed the roles of alpha, beta to line up better with the computations later on, and reversed the sign convention on epsilon.
%%%%%%%%%%%
%	\item Suppose $\alpha$ is long and $\beta$ short. Then we can write $\alpha = \eps 2 \alpha_i, \beta = - \eps \alpha_i + \omega \alpha_j$ for some $\eps = \pm 1, \omega = \pm 1$ and $i \neq j$, and
%\begin{align*}
%	N_{11}^{\alpha\beta}(u,v) &= - \omega u v_{-c_{ij}}  \\
%	N_{12}^{\alpha\beta}(u,v) &= - \eps \omega uv \ov v
%\end{align*}
%	for all $u \in V_\alpha(R), v \in V_\beta(R)$, where
%\[
%c_{ij} =
%	\begin{cases}
%		1 & i < j \\
%		-1 & i > j
%	\end{cases}
%\]
%%%%%%%%%%%

Next, as we indicated in \S\ref{S-Introduction}, a key step in the proof of our main result involves restricting a given representation $\rho : G(k) \to \mathrm{GL}_m(K)$ to the subgroup $G_0(k) = \mathrm{Sp}_{2n}(k).$ Our choice of $H$ makes this split subgroup simple to describe. It is
%$For the argument, it will be useful to have an explicit description of $G_0(k)$ as a subgroup of $G(k).$ Recall that in our definition of $G = \SU_{2n}(L,h)$, we fixed a basis $\mathcal{B} = \{e_1, e_2, \dots, e_{2n-1}, e_{2n} \}$ of $V = L^{2n}$ with respect to hich the matrix of $h$ is
%\[
%H =
%	\begin{pmatrix}
%		0 & -1   \\
%		1 & 0  \\
%		& & \ddots   \\
%		& & & 0 & -1 \\
%		& & & 1 & 0
%	\end{pmatrix}.
%\]
%Replacing $\mathcal{B}$ with the basis $\mathcal{B}' = \{\sqrt{d} e_1, e_2, \dots, \sqrt{d} e_{2n-1}, e_{2n} \}$, the matrix of $h$ becomes
%$$
%H' = \begin{pmatrix}
%		0 & -\sqrt{d}   \\
%		\sqrt{d} & 0  \\
%		& & \ddots   \\
%		& & & 0 & -\sqrt{d} \\
%		& & & \sqrt{d} & 0
%	\end{pmatrix}.
%$$
%Set $H'' = \frac{1}{\sqrt{d}} H'.$ We then have
%$$
%G(R) = \{X \in \mathrm{SL}_{2n}(R_L) \mid X^* H'' X = H'' \}
%$$
%for any $k$-algebra $R$, and hence
%\begin{equation}\label{E-SplitSubgroup}
%G(k) \cap \mathrm{GL}_{2n}(k) = \{ X \in \mathrm{SL}_{2n}(k) \mid X^t H'' X = H''\} = \mathrm{Sp}_{2n}(k) = G_0(k).
%\end{equation}
\begin{equation}\label{E-SplitSubgroup}
G(k) \cap \mathrm{GL}_{2n}(k) = \{ X \in \mathrm{SL}_{2n}(k) \mid X^t H X = H \} = \mathrm{Sp}_{2n}(k) = G_0(k).
\end{equation}
Now, recall that if $\beta \in \Phi_k$ is a short root, then $V_{\beta}(k) = L = k \oplus k \sqrt{d}.$ Our calculation then shows that $G_0(k)$ is the subgroup of $G(k) = E(k)$ generated by the images of the maps $X_\alpha(k) : V_{\alpha}(k) \to G(k)$ for long roots $\alpha$ and by the images of the maps $X_{\beta}(k) : V_{\beta}(k) \to G(k)$ restricted to the first component for long roots $\beta.$

For later use, following Steinberg (see \cite[Ch. 3]{Steinberg}), we define the elements
\begin{equation}
\label{SteinbergElements}
	w_\alpha(v) = X_\alpha(v) \cdot X_{-\alpha}(-v \inv) \cdot X_\alpha(v)  \ \ \ \text{and} \ \ \ h_\alpha(v) = w_\alpha(v) \cdot w_\alpha(1) \inv
\end{equation}
for $\alpha \in \Phi_k$ and $v \in V_{\alpha}(k)^{\times}$ (thus, in particular, $w_{\alpha}(1), h_{\alpha}(1) \in G_0(k)$). Then we have the following analogue of Steinberg's relation (R7)
\begin{equation}\label{SteinbergElements1}
w_\alpha(1) \cdot X_\beta(v) \cdot w_\alpha(1) \inv = X_{w_\alpha \beta}(\varphi v),
\end{equation}
where $\varphi v = \pm v_{\pm 1}$ (so that $\varphi^2 = \mathrm{Id}$) and $w_{\alpha} \beta$ denotes the action of the corresponding Weyl group element $w_\alpha$ on the root $\beta.$

%\subsubsection{Split subgroup}
%\label{split subgroup}
%In terms of the presentation above, one particular $k$-split subgroup $G_0 = \Sp_{2n}$ is the subgroup generated by the images of the $X_\alpha$ for long $\alpha$ and the images of $X_\beta$ restricted to the first component, for $\beta$ short. That is, if $\beta$ is short so that $V_\alpha(R) = R_L$, restrict $X_\beta:R_L \to G(R)$ to $R = R \otimes_k k$; the resulting images in $G(R)$ are generators of $G_0(R)$. Henceforth $G_0$ refers to this particular split subgroup.

%As a split group, $G_0$ satisfies the relations (R) of Steinberg \cite[Chapter 3]{Steinberg}. To describe these, for $\alpha \in \Phi_k$ and $v \in V_\alpha(k)^\times$, define
%\begin{align*}
%	w_\alpha(v) &= X_\alpha(v) \cdot X_{-\alpha}(-v \inv) \cdot X_\alpha(v) \\
%	h_\alpha(v) &= w_\alpha(v) \cdot w_\alpha(1) \inv
%\end{align*}
%We will later make particular use of Steinberg's relation (R7), which in our notation is
%\[
%	w_\alpha(1) \cdot X_\beta(t) \cdot w_\alpha(1) \inv = X_{w_\alpha \beta}(\varphi t)
%\]
%where $\alpha,\beta \in \Phi_k, t \in k,$ and $\varphi t = \pm t$. This extends to $G(k)$ replacing $t \in k$ by $v \in V_\alpha(k)$, and possibly conjugating $v$. That is, for $\alpha, \beta \in \Phi_k$ and $v \in V_\alpha(A)$,
%\[
%	w_\alpha(1) \cdot X_\beta(v) \cdot w_\alpha(1) \inv = X_{w_\alpha \beta}( \varphi v)
%\]
%where $\varphi v = \pm v_{\pm 1}$. Notably, $\varphi^2 = \Id$.

\subsection{Steinberg groups}

 In this subsection, we recall Stavrova's \cite{Stavrova} generalization of the usual Steinberg group (which, in turn, is inspired by a construction of Deodhar \cite{Deodhar}) and establish several properties that will be needed in subsequent sections. Although Stavrova works in the general context of reductive group schemes, for the sake of concreteness, we will restrict ourselves to the case $G = \SU_{2n}(L,h).$

As in the classical setting, for a $k$-algebra $R$, the generalized Steinberg group $\wt{G}(R)$ is the abstract group generated by symbols $\wt{X}_{\alpha}(v)$, for $\alpha\in \Phi_k$ and $v \in V_{\alpha}(R)$, subject to the relations in Theorem \ref{root subgroup maps} written in terms of the $\wt{X}_{\alpha}$. More precisely, we have the following.

\begin{definition}
 Let $R$ be a $k$-algebra. We define $\wt{G}(R)$ to be the group generated by symbols $\wt{X}_{\alpha}(v)$, for all $\alpha\in \Phi_k$ and all $v \in V_{\alpha}(R)$, subject to the following relations:
\begin{enumerate}
	\item[(R1)] For $\alpha \in \Phi_k$ and $v,w \in V_\alpha(R)$,
	\[
		\wt X_\alpha(v) \cdot \wt X_\alpha(w) = \wt X_\alpha(v+w);
	\]
	\item[(R2)] For $\alpha, \beta \in \Phi_k$ such that $\alpha \neq \pm \beta$, and for all $u \in V_\alpha(R), v \in V_\beta(R)$,
	\[
		\Big[ \wt X_\alpha(u), \wt X_\beta(v) \Big] = \prod_{\substack{{i,j \ge 1} \\ i\alpha+j\beta \in \Phi_k}} \wt X_{i\alpha+j\beta} \Big( N_{ij}^{\alpha \beta}(u,v) \Big);
	\]
	where $N_{ij}^{\alpha \beta}$ are the same maps as in Theorem \ref{root subgroup maps}(3).
\end{enumerate}
\end{definition}

\noi It is clear that this construction is functorial in $R$. Furthermore, for every $k$-algebra $R$, we have a natural surjective homomorphism
\[
	\pi_{R}:\wt{G}(R) \to E(R) \qquad \wt X_\alpha(v) \mapsto X_\alpha(v).
\]
The main result of this subsection is the following statement, which partially extends \cite[Theorem 1.3]{Stavrova}.

\begin{proposition}\label{central kernel}
Suppose $R$ is a $k$-algebra that is a finite product of local $k$-algebras. Then $\ker \pi_R$ is a central subgroup of $\wt{G}(R)$.
\end{proposition}

\noi We will give the proof of Proposition \ref{central kernel} below, after several preliminary observations.

\begin{lemma}
\label{product of other generators}
Any generator $\wt X_\alpha(v)$ of $\wt{G}(R)$ can be written as a product of generators $\wt X_{\gamma_i}(u_i)$ with $\gamma_i \neq \alpha$. Furthermore, $\wt X_\alpha(v)$ is contained in the commutator subgroup $[\wt{G}(R), \wt{G}(R)].$ In particular, $\wt{G}(R)$ and $E(R)$ are perfect groups. %(i.e. they coincide with their commutator subgroups.
\end{lemma}
\begin{proof}
First, suppose $\alpha$ is a long root. Then we can write $\alpha = \gamma_1 + \gamma_2$ for appropriate short roots $\gamma_1$ and $\gamma_2$, and it follows from (R2) that
\[
	\Big[ \wt X_{\gamma_1}(u_1), \wt X_{\gamma_2}(u_2) \Big] = \wt X_\alpha \Big( N_{11}^{\gamma_1 \gamma_2}(u_1, u_2) \Big)
\]
Since $N_{11}^{\gamma_1 \gamma_2}$ is surjective according to Lemma \ref{N is surjective}, we can find $u_1 \in V_{\gamma_1}(R)$ and $u_2 \in V_{\gamma_2}(R)$ so that $N_{11}^{\gamma_1 \gamma_2}(u_1, u_2) = v$, which yields our claim in this case.

Next, suppose $\alpha$ is short. Then we can write $\alpha = \gamma_1 + \gamma_2$ for an appropriate long root $\gamma_1$ and short root $\gamma_2$. Hence, by (R2) we have
\[
	\Big[ \wt X_{\gamma_1}(u_1), \wt X_{\gamma_2}(u_2) \Big] = \wt X_{\alpha} \Big( N_{11}^{\gamma_1 \gamma_2}(u_1, u_2) \Big) \cdot \wt X_{\gamma_1 + 2 \gamma_2} \Big( N_{12}^{\gamma_1 \gamma_2} (u_1, u_2) \Big)
\]
Again by Lemma \ref{N is surjective} we can choose $u_1, u_2$ so that $N_{11}^{\gamma_1 \gamma_2}(u_1 ,u_2) = v$. Multiplying both sides by $\wt X_{\gamma_1 + 2 \gamma_2} \Big( N_{12}^{\gamma_1 \gamma_2} (u_1, u_2)\Big)^{-1}$ then yields our first claim. Furthermore, since $\gamma_1 + 2\gamma_2$ is a long root, the preceding case shows that $\wt X_{\gamma_1 + 2 \gamma_2} \Big( N_{12}^{\gamma_1 \gamma_2} (u_1, u_2)\Big)^{-1}$ is contained in the commutator subgroup, and hence so is $\wt{X}_{\alpha}(v).$

Thus, since all generators of $\wt{G}(R)$ are contained in $[\wt{G}(R), \wt{G}(R)]$, it follows that $\wt{G}(R)$ is perfect. Since the natural map $\pi : \wt{G}(R) \to E(R)$ is surjective, we conclude that $E(R)$ is perfect as well.
\end{proof}

\begin{lemma}
\label{Steinberg products}
For any $k$-algebras $R_1, \dots, R_n$, we have $\wt{G}(R_1 \times \cdots  \times R_n) = \wt{G}(R_1) \times \cdots \times \wt{G}(R_n)$, i.e. the Steinberg group commutes with finite products.
\end{lemma}
\begin{proof}
By induction, it suffices to show that $\wt{G}(A \times B) \iso \wt{G}(A) \times \wt{G}(B)$ for any $k$-algebras $A$ and $B$. We will do this by an argument similar to the one used in the proof of \cite[Lemma 2.12]{Stein}. First, we note that the projections of $A \times B$ onto its components induce a surjective group homomorphism
\[
	p:\wt{G}(A \times B) \to \wt{G}(A) \times \wt{G}(B), \qquad \wt X_\alpha(v) = \wt X_\alpha(v_1, v_2) \mapsto \Big( \wt X_\alpha(v_1), \wt X_\alpha(v_2) \Big)
\]
Next we define a map in the reverse direction by
\begin{align*}
	s:\wt{G}(A) \times \wt{G}(B) &\to \wt{G}(A \times B) \\
	\Big( \wt X_\alpha(v_1), 1 \Big) &\mapsto \wt X_\alpha(v_1, 0) \\
	\Big( 1, \wt X_\alpha(v_2) \Big) &\mapsto \wt X_\alpha(0, v_2)
\end{align*}
(This defines $s$ on a set of generators for $\wt{G}(A) \times \wt{G}(B)$, and we then extend it to the whole group by multiplicativity.) It follows immediately from the definitions that $p \circ s$ and $s \circ p$ are the respective identity maps on generating sets. Thus, it remains to show that $s$ is a homomorphism by verifying that $s$ takes all the defining relations in $\wt G(A) \times \wt G(B)$ to relations in $\wt G(A \times B)$. We need to check that $s$ preserves three kinds of relations:
\begin{enumerate}[(i)]
	\item The defining relations of $\wt G(A)$ applied to the generators $\lp \wt X_\alpha(v_1), 1 \rp$.
	\item The defining relations of $\wt G(B)$ applied to the generators $\lp 1, \wt X_\alpha(v_2) \rp$.
	\item $\left[ \lp \wt X_\alpha(a), 1 \rp, \lp 1, \wt X_\beta(b) \rp \right] =1$ for all $\alpha, \beta \in \Phi_k$ and all $a \in V_\alpha(A), b \in V_\beta(B)$.
\end{enumerate}
It is clear that $s$ preserves (i) and (ii), so it remains to show that $s$ preserves (iii). First consider the case $\alpha \neq - \beta$. Let $a \in V_\alpha(A)$ and $b \in V_\beta$. Then
\begin{align*}
	s \left[ \lp \wt X_\alpha(a), 1 \rp, \lp 1, \wt X_\beta(b) \rp \right] &= \left[ \wt X_\alpha(a,0), \wt X_\beta(0,b) \right] = \prod_{\substack{i,j \ge 1 \\ i \alpha + j \beta \in \Phi_k}} X_{i \alpha + j \beta} \lp N_{ij}^{\alpha \beta} \big( (a,0), (0,b) \big) \rp
\end{align*}
Since $N_{ij}^{\alpha \beta}$ is homogeneous of degree $i$ in the first argument and of degree $j$ in the second argument, it is at least linear in each argument. So each term of $N_{ij}^{\alpha \beta}\big((a,0), (0,b) \big)$ includes a factor of $(a,0) \cdot (0,b) = (0,0) = 0 \in A \times B$. Hence $N_{ij}^{\alpha \beta} \big( (a,0), (0,b) \big) = 0$ and the product is trivial, so $s$ preserves (iii) when $\alpha \neq - \beta$.

Now consider relation (iii) with $\alpha = - \beta$. For any $b \in B$, by Lemma \ref{product of other generators}, the element $\wt X_{-\alpha}(0,b) \in \wt G(0 \times B) \subset \wt G(A \times B)$ can be written as a product
\[
	\wt X_{-\alpha}(0,b) = \prod_i \wt X_{\gamma_i} (0, u_i)
\]
where $\gamma_i \neq -\alpha$ for all $i$, with each $u_i \in B$. Thus, we have
\begin{align*}
	s \left[ \lp \wt X_\alpha(a), 1 \rp, \lp 1, \wt X_{-\alpha}(b) \rp \right] &= \left[ \wt X_\alpha(a,0), \wt X_{-\alpha}(0,b) \right] = \left[ \wt X_\alpha(a,0), \prod_i \wt X_{\gamma_i} (0,u_i) \right]
\end{align*}
Since $\gamma_i \neq -\alpha$, by the previous case, $\wt X_\alpha(a,0)$ commutes with each factor $\wt X_{\gamma_i}(0,u_i)$. Consequently, it commutes with the whole product, hence the last commutator vanishes. This shows that $s$ preserves (iii) when $\alpha = - \beta$, completing the proof.
\end{proof}

\noi We now turn to:

\vskip1mm

\noindent {\it Proof of Proposition \ref{central kernel}.}
Suppose $R = R_1 \times \cdots \times R_n$, where $R_1, \dots, R_n$ are local $k$-algebras. Since the elementary subgroup and the Steinberg group both commute with finite products (by Remark \ref{G=E} and Lemma \ref{Steinberg products}, respectively), it follows that
$$
\pi_R = \prod_{i=1}^n \pi_{R_i},
$$
and hence
\[
	\ker \pi_R = \prod_{i=1}^n \ker \pi_{R_i} \subset \prod_{i = 1}^n \wt{G}(R_i) \iso \wt{G}(R).
\]
Applying \cite[Theorem 1.3]{Stavrova} to each local factor $R_i$, we see that $\ker \pi_{R_i}$ is central in $\wt{G}(R_i)$. Consequently, $\ker \pi_R$ is central in $\wt{G}(R)$, as claimed. $\Box$

\section{Constructing the algebraic ring $A$ and lifting $\rho$ to a representation of the Steinberg group}\label{S-AlgebraicRing}

Let $G = \SU_{2n}(L,h)$ and set $G_0(k) = \mathrm{Sp}_{2n}(k)$ to be the split subgroup of $G(k)$ described in (\ref{E-SplitSubgroup}). The purpose of this section is twofold. First, we associate an algebraic ring $A$ to an abstract representation $\rho \colon G(k) \to \GL_m(K).$ Then we show that $\rho$ can be lifted to a representation $\wt{\sigma}$ of the Steinberg group $\wt{G}(A).$

To fix notations, given an abstract representation $\rho \colon G(k) \to \GL_m(K)$ (with $K$ an algebraically closed field of characteristic 0), we set $H = \ov{\rho(G(k))}$ to be the Zariski closure of the image of $\rho.$ Furthermore, we let $\rho_0 \colon G_0(k) \to \GL_m(K)$ denote the restriction $\rho \vert_{G_0(k)}.$

By \cite[Theorem 3.1]{IR}, we can associate to $\rho_0$ an algebraic ring $A$, together with a ring homomorphism $f \colon k \to A$ with Zariski-dense image, as follows. Recall that if $\alpha \in \Phi_k$ is a long root, then $V_{\alpha}(k) = k$, whereas if $\alpha$ is short, then $V_{\alpha}(k) = L = k \oplus k \sqrt{d}.$
%Let $G = \SU_{2n}(L,h)$, let $\rho:G(k) \to \GL_m(K)$ be an abstract representation, and set $H = \ov{\rho(G(k))}$. In this section we construct the algebraic ring $A$ of Theorem \ref{T-MainTheorem}, as well as the ring homomorphism $f:k \to A$. Then we lift $\rho$ to a representation $\wt \sig$ of $\wt G(A)$.
%Let $G_0 = \Sp_{2n}$ be the previously described split subgroup of $G$, generated by restrictions of the maps $X_\alpha$. We repeat the construction of the algebraic ring used for \cite[Theorem 3.1]{IR}, applied to $\rho|_{G_0(k)}$. Recall that $k$ sits inside $V_\alpha(k)$.
For $\alpha \in \Phi_k$, we define
\[
	A_\alpha = \ov{\rho(X_\alpha(k))}
\]
and
\[
	f_\alpha \colon k \to A_\alpha, \qquad u \mapsto \rho_0(X_\alpha(u)).
\]
Recall that for $\alpha$ short, the root subgroup $X_\alpha(V_\alpha(k))$ has dimension 2, and in this case, 
%and note that in this case $A_\alpha$ is not the closure of the image of $X_\alpha(V_\alpha(k))$; rather 
$A_\alpha$ is the closure of the image of the one-dimensional subgroup $X_\alpha(k) \subset X_\alpha(V_\alpha(k))$, arising from the natural embedding $k \into L = V_\alpha(k)$.

As shown in \cite[Theorem 3.1]{IR}, each $A_{\alpha}$ has the structure of an algebraic ring (i.e. an affine algebraic variety with a ring structure defined by regular maps) ---  we recall that the addition operation is obtained simply by restricting matrix multiplication in $H$ to $A_{\alpha}$, whereas the multiplication operation is defined using the Steinberg commutator relations. Moreover, for any $\alpha, \beta \in \Phi_k$, there exists an isomorphism $\pi_{\alpha \beta} \colon A_{\alpha} \to A_{\beta}$ of algebraic rings such that $\pi_{\alpha \beta} \circ f_{\alpha} = f_{\beta}.$ We denote this common algebraic ring by $A$, and, for each $\alpha \in \Phi_k$, we fix an isomorphism of algebraic rings $\pi_{\alpha} \colon A \to A_{\alpha}$ such that $\pi_{\alpha \beta} \circ \pi_\alpha = \pi_\beta$. Then, by construction, we have a ring homomorphism $f \colon k \to A$ with Zariski-dense image such that $\pi_{\alpha} \circ f = f_{\alpha}$ for all $\alpha \in \Phi_k$, as well as (injective) regular maps $\psi_{\alpha}^1 \colon A \to H$ satisfying
$$
(\psi_{\alpha}^1 \circ f)(u) = (\rho_0 \circ X_{\alpha})(u) \ \ \ \text{for all} \ u \in k.
$$

%It is immediate that $f_\alpha$ has Zariski-dense image in $A_\alpha$. In the case $n=2$, after passing through a root system isomorphism $\Csf_2 \iso \mathsf{B}_2$, these are the same as the rings $A_\alpha, A_\beta$ of \cite[Theorem 3.1, Case II]{IR}. By that theorem, each $A_\alpha$ has the structure of an algebraic ring, and each is isomorphic (as an algebraic ring) to an algebraic ring $A$. Also by the theorem, there is an abstract ring homomorphism $f:k \to A$ and (injective) regular maps $\psi_\alpha^1:A \to H$ such that $\psi_\alpha^1 \circ f = \rho \circ X_\alpha|_k$ for all $\alpha \in \Phi_k$. For each $\alpha \in \Phi_k$, fix an isomorphism of algebraic rings $\pi_\alpha:A \to A_\alpha$ such that $\pi_\alpha \circ f = f_\alpha$.

%Addition in $A_\alpha$ is matrix multiplication inside $\GL_m(K)$, but multiplication in $A_\alpha$ is somewhat involved to describe; see \cite[Lemma 3.3]{IR} and the ensuing discussion.

\begin{remark}
\label{properties of A}
Since $k$ and $K$ are both fields of characteristic 0, it follows from Lemma 2.13 and Proposition 2.14 in \cite{IR} that $A$ is a finite-dimensional $K$-algebra. In particular, $A$ is connected and satisfies the hypotheses of Proposition \ref{central kernel}.
%Using the fact that $\characteristic k = 0$, by \cite[Lemmas 2.8, 2.9, 2.13, 2.20]{IR} the algebraic ring $A$ above is semilocal, connected, Artinian, and a direct sum of local Artinian rings.
\end{remark}

The algebraic ring $A$ plays a pivotal role in the proof of the main results of \cite{IR}. It turns out that $A$ also suffices for the analysis of the representation $\rho.$ The precise statement that is needed in our context will be given in Proposition \ref{existence of A} below. First, we make the following construction. Given a short root $\alpha \in \Phi_k$, let us define
\[
	B_\alpha = \ov{\rho \lp X_\alpha \lp k \sqrt d \rp \rp}
\]
and
\[
	g_\alpha:k \to B_\alpha, \qquad u \mapsto \rho \lp X_\alpha \lp u \sqrt d \rp \rp.
\]
We then have the next statement.
%In \cite[Lemma 3.3]{IR}, they show that for a particular pair of roots $\alpha,\beta$ with $\alpha$ short, $\beta$ long, and $\beta - \alpha \in \Phi_k$, there is an isomorphim $\pi_{\alpha \beta}:A_\alpha \to A_\beta$ such that $\pi_{\alpha \beta} \circ f_\alpha = f_\beta$. In particular, $\pi_{\alpha \beta}$ is given by
%\begin{equation}
%\label{original compatibility isomorphism}
%	\pi_{\alpha \beta}(x) = \big[ x, f_{\beta-\alpha} (1/2) \big]
%\end{equation}
%We extend this in the following lemma. Note that for such a pair $\alpha, \beta$ as above, they span a copy of $\Csf_2$ sitting inside $\Phi_k$ consisting of the roots $\lb \pm \alpha, \pm \beta, \pm (\beta - \alpha), \pm (2\alpha-\beta) \rb$, where $\pm \alpha$ and $\pm (\beta-\alpha)$ are short, and the other roots are long. All of the Chevalley commutator relations in the following lemma involve only this copy of $\Csf_2$ inside $\Phi_k$.

\begin{lemma}
\label{compatibility isomorphism}
Take the short root $\alpha = \alpha_1 - \alpha_3$ and the long root $\beta = -2\alpha_3$. Then there is an isomorphism of algebraic varieties $\pi \colon B_\alpha \to A_\beta$ such that $\pi \circ g_\alpha = f_\beta$.
\end{lemma}
\begin{proof}
(cf. \cite[Lemma 3.3]{IR})
%We perform a similar computation to that in \cite[Lemma 3.3]{IR}.
By Lemma \ref{N is surjective}(1), we have $N_{11}^{\alpha,\beta-\alpha}(u,v) = - \Tr(uv)$. Define a regular map $\pi \colon B_\alpha \to H$ by
\[
	\pi(x) = \left[ x, g_{\beta-\alpha} \lp \frac{-1}{2d} \rp \right].
\]
%(Compare with equation (\ref{original compatibility isomorphism}).) Note $2d$ is invertible as $\characteristic k = 0$.
%Note that $\pi$ is a regular map.
%This commutator occurs inside $\GL_m(K)$ and multiplication is regular, so $\pi$ is regular.
Now let $s \in k$. Then
\begin{align*}
	\pi \circ g_\alpha(s) &= \left[ \rho \circ X_\alpha \lp s \sqrt d \rp , \rho \circ X_{\beta-\alpha} \lp \frac{- \sqrt{d}}{2 d} \rp \right] = \rho  \left[ X_\alpha \lp s \sqrt d \rp, X_{\beta-\alpha} \lp \frac{-1}{2 \sqrt{d}} \rp \right] \\
	&= \rho \circ X_\beta \lp N_{11}^{\alpha,\beta-\alpha} \lp s \sqrt d, \frac{-1}{2 \sqrt{d}} \rp \rp  = \rho \circ X_\beta \lp - \Tr \lp \frac{- s}{2} \rp \rp  = \rho \left( X_\beta(s) \right) = f_\beta(s)
\end{align*}
This shows that $\pi \circ g_\alpha = f_\beta$.
%and that $\pi$ maps $g_\alpha(k)$ into $f_\beta(k)$.
Since $\pi$ is regular, it then follows that $\pi (B_{\alpha}) \subset A_{\beta}$. It remains to show that $\pi$ is invertible. First, using Lemma \ref{N is surjective}(2), we obtain
\begin{align*}
	N_{11}^{\beta, \alpha-\beta}(v,u) &= - uv \\
	N_{12}^{\beta, \alpha-\beta}(v,u) &= vu \ov u
\end{align*}
Now let $h = h_{2\alpha-\beta}(1/2)$ be the element introduced in (\ref{SteinbergElements}) and define
\[
	\nu:A_\beta \to B_\alpha, \qquad \nu(y) = \rho ( h ) \cdot \left[ y, g_{\alpha-\beta}(-1) \right] \cdot \left[ y, g_{\alpha-\beta}( 1) \right] \inv \cdot \rho (h) \inv.
\]
It is clear that $\nu$ is a regular map; we claim it is an inverse for $\pi$. Let $t \in k$. Using the commutator relation in Theorem \ref{root subgroup maps}(3), we have
\begin{align*}
	\left[ X_{\beta}(t), X_{\alpha-\beta} \lp -\sqrt d \rp \right] &= X_{\alpha} \Big( N_{11}^{\beta,\alpha-\beta}\lp t, - \sqrt d \rp \Big) \cdot X_{2\alpha-\beta} \Big( N_{12}^{\beta,\alpha-\beta} \lp t, - \sqrt d \rp \Big) \\
	&= X_\alpha \Big( t \sqrt{d} \Big) \cdot X_{2\alpha-\beta} \Big( -t d \Big).
\end{align*}
Thus
\begin{align*}
	&\left[ X_{\beta}(t),  X_{\alpha-\beta} \lp - \sqrt{d} \rp \right] \cdot \left[ X_{\beta}(t),  X_{\alpha-\beta}\lp \sqrt{d} \rp \right] \inv \\
	&=X_\alpha \Big( t \sqrt{d} \Big) \cdot \cancelto{}{X_{2\alpha-\beta} \Big(-  t d \Big)} \cdot \cancelto{}{X_{2\alpha-\beta} \Big( - t d \Big) \inv} \cdot X_\alpha \Big( - t \sqrt{d}  \Big) \inv = X_\alpha\lp 2 t \sqrt{d} \rp.
\end{align*}
We also have the relation
\[
	h \cdot X_\alpha(2v) \cdot h \inv = X_\alpha \lp v \rp
\]
for all $v \in L$. Putting everything together, we obtain
\begin{align*}
	\nu \circ f_\beta(t)  &= \rho  \lp h \cdot \left[  X_{\beta}(t) , X_{\alpha-\beta} \lp -\sqrt{d} \rp \right] \cdot \left[  X_{\beta}(t) , X_{\alpha-\beta}\lp \sqrt{d} \rp \right] \inv \cdot h \inv \rp \\
	&=  \rho  \lp h \cdot X_\alpha\lp 2 t \sqrt{d} \rp \cdot h \inv \rp = \rho \lp X_\alpha \lp t \sqrt{d} \rp \rp = g_\alpha(t).
\end{align*}
Thus, $\nu \circ \pi$ and $\pi \circ \nu$ are the respective identity maps on dense subsets of $A_\beta$ and $B_\alpha$. Since they are regular, it follows that they are the respective identities on the whole space, so $\nu$ is the inverse of $\pi$ as claimed.
\end{proof}

\begin{remark}
Although we fixed the roots $\alpha$ and $\beta$ in the preceding argument for the sake of concreteness, essentially the same calculations can be carried for any short root $\alpha$ and long root $\beta$ such that $\alpha - \beta$ is a root (with appropriate modifications to the definitions of $\pi$ and $\nu$, depending on the signs arising in computing $N_{11}$ and $N_{12}$).
\end{remark}

Next, we make the following observation, which will streamline the proof of Proposition \ref{existence of A} by allowing us to consider just a single root of each length. In the statement, we refer to the group schemes $V_{\alpha}$ introduced in Theorem \ref{root subgroup maps}. Given a $k$-algebra homomorphism $f \colon k \to A$, we denote by $V_{\alpha}(f) \colon V_{\alpha}(k) \to V_{\alpha}(A)$ the associated group homomorphism (note that if 
%Note that if 
$\alpha$ is a long root, then $V_{\alpha}(f)$ can be identified with $f$, and if $\alpha$ is short, then $V_{\alpha}(f)$ is the map $a + b \sqrt{d} \mapsto f(a) + f(b) \sqrt{d}$ --- in particular, if $\alpha$ and $\beta$ have the same length, then the homomorphisms $V_\alpha(f)$ and $V_\beta(f)$ coincide).
%the corresponding root space is 1-dimensional, and $V_{\alpha}(f)$ can be identified with $f.$ On the other hand, if $\alpha$ is short, then $V_{\alpha}(f)$ is the map
%We also note that if $\alpha$ and $\beta$ are long roots, then both $V_{\alpha}(f)$ and $V_{\beta}(f)$ are simply $f$, whereas if they are short, then the maps are both
%\[
	%a + b \sqrt{d} \mapsto f(a) + f(b) \sqrt{d}
%\]
%from $k \oplus k \sqrt{d}$ to itself. In particular, if $\alpha$ and $\beta$ have the same length, then the homomorphisms $V_\alpha(f)$ and %$V_\beta(f)$ coincide.
%are the same map.

\begin{lemma}
\label{Weyl group reduction}
Let $\alpha,\beta \in \Phi_k$ be roots of the same length. Then
\begin{enumerate}
	\item There exists $w \in E(k)$ such that for all $v \in V_\alpha(k)$, we have $w \cdot X_\alpha(v) \cdot w \inv = X_\beta( \varphi v)$, where $\varphi v = \pm v_{\pm 1}$.
	\item Let $f \colon k \to A$ be a $k$-algebra homomorphism. Suppose there exists a regular map $\psi_\alpha \colon V_\alpha(A) \to H$ such that $\psi_\alpha \circ V_\alpha(f) = \rho \circ X_\alpha$. Let $w, \varphi$ be as in (1), and define $\psi_\beta \colon V_{\beta}(A) \to H$ by
	\[
		\psi_\beta(v) = \rho(w) \cdot \psi_\alpha(\varphi v) \cdot \rho(w) \inv
	\]
	Then $\psi_\beta$ is regular and satisfes $\psi_\beta \circ V_\beta(f) = \rho \circ X_\beta$.
\end{enumerate}
\end{lemma}
\begin{proof}
(1) It is well-known that the Weyl group $W$ acts transitively on roots of the same length, so there exists $\wt w \in W$ such that $\wt w \alpha = \beta$. Write $\wt w$ as a product of simple reflections, $\wt w = w_{\gamma_1} \cdots w_{\gamma_n}$ for roots $\gamma_i \in \Phi_k$. Using the relation (\ref{SteinbergElements1}), we have
\[
	w_{\gamma_i}(1) \cdot X_\alpha(v) \cdot w_{\gamma_i}(1) \inv = X_{w_{\gamma_i} \alpha}(\varphi_i v)
\]
where $\varphi_i v = \pm v_{\pm 1}$. Let $w = w_{\gamma_1}(1) \cdots w_{\gamma_n}(1)$ and $\varphi = \varphi_1 \cdots \varphi_n$. Then repeatedly applying the above relation yields
\[
	w \cdot X_\alpha(v) \cdot w \inv = X_{w_{\gamma_1} \cdots w_{\gamma_n} \alpha} ( \varphi_1 \cdots \varphi_n v) = X_{\wt w \alpha} ( \varphi v) = X_\beta(\varphi v)
\]
Since each $\varphi_i$ is a composition of negation and conjugation, $\varphi v = \pm v_{\pm 1}$.

\vskip2mm

\noi (2) This is proved by a direct calculation using the result of part (1), the fact that $\varphi^2 = \Id$, and the assumption that $\psi_\alpha \circ V_\alpha(f) = \rho \circ X_\alpha$. Namely, let $v \in V_\beta(A)$. Then
\begin{align*}
	\psi_{\beta} \circ V_\beta(f)(v) &= \rho(w) \cdot \psi_{\alpha}( \varphi \circ V_\beta(f)(v) ) \cdot \rho(w) \inv = \rho(w) \cdot \psi_{\alpha} \circ V_\alpha(f)( \varphi v) \cdot \rho(w) \inv = \\
	&= \rho(w) \cdot \rho \circ X_{\alpha}( \varphi v) \cdot \rho(w) \inv = \rho \Big( w \cdot X_{\alpha}(\varphi v) \cdot w \inv \Big) = \rho \Big( X_{\beta}(\varphi^2 v) \Big) = \rho \circ X_{\beta}(v)
\end{align*}
\end{proof}

\noi We now come to one of the main statements of this section.

\begin{proposition}
\label{existence of A}
Let $G = \SU_{2n}(L,h)$ and let $\rho:G(k) \to \GL_m(K)$ be an abstract representation, with $K$ an algebraically closed field of characteristic. Set $H = \ov{\rho(G(k))}$. There exists a finite-dimensional $K$-algebra $A$, a ring homomorphism $f \colon k \to A$ with Zariski-dense image, and for each $\alpha \in \Phi_k$, a regular map $\psi_\alpha \colon V_{\alpha}(A) \to H$ such that $\psi_\alpha \circ V_\alpha(f) = \rho \circ X_\alpha$.
\end{proposition}
\begin{proof}
We take $A$ and $f \colon k \to A$ to be the algebraic ring and ring homomorphism constructed from the restriction $\rho \vert_{G_0(k)}$ using \cite[Theorem 3.1]{IR}, as described at the beginning of this section. Recall that
%By \cite[Theorem 3.1]{IR} applied to the restriction $\rho|_{G_0(k)}$, there exists an algebraic ring $A$, a ring homomorphism $f:k \to A$ with Zariski-dense image, and
for each $\alpha \in \Phi_k$, there is a regular map $\psi_\alpha^1:A_\alpha \to H$, and an isomorphism $\pi_\alpha \colon A \to A_\alpha$ such that $\pi_\alpha \circ f = f_\alpha$ and $\psi_\alpha^1 \circ V_\alpha(f)|_k = \rho \circ X_\alpha|_k$.

First, we note that by Lemma \ref{Weyl group reduction}, it suffices to construct $\psi_\alpha \colon V_\alpha(A) \to H$ satisfying $\psi_\alpha \circ V_\alpha(f) = \rho \circ X_\alpha$ for a single root of each length. If $\alpha$ is a long root, then, since the corresponding root space is 1-dimensional, we can simply set $\psi_\alpha = \psi_\alpha^1$.

%Since $\alpha$ is short, $V_\alpha(f)|_k = V_\alpha(f) = f$ and $X_\alpha|_k = X_\alpha$, so $\psi_\alpha \circ V_\alpha(f) = \rho \circ X_\alpha$. This suffices for the long roots.

Now let us consider the short root $\alpha = \alpha_1 - \alpha_3$. Then, taking $\beta = -2 \alpha_3$,
%$\beta = -2\alpha_3$. Then $\beta$ is long and $\beta - \alpha \in \Phi_k$, so by
Lemma \ref{compatibility isomorphism} yields
an isomorphism of algebraic varieties $\pi \colon B_\alpha \to A_\beta$ such that $\pi \inv \circ f_{\beta} = g_\alpha$. Define
\[
	\psi_\alpha^2 \colon A \to H, \qquad \psi_\alpha^2 = \iota_B \circ \pi \inv \circ \pi_{\beta},
\]
where $\pi_{\beta} \colon A \to A_{\beta}$ is the previously fixed isomorphism satisfying $f_{\beta} = \pi_{\beta} \circ f$, and $\iota_B \colon B_\alpha \into H$ is the natural inclusion. Using the identification $V_{\alpha}(A) = A_L \simeq A^2$, $v = v_1 + v_2 \sqrt{d} \mapsto (v_1, v_2)$, we define
\[
	\psi_\alpha:V_\alpha(A) \to H, \qquad v= (v_1, v_2) \mapsto \psi_\alpha^1(v_1) \cdot \psi_\alpha^2(v_2).
\]
Note that for any $u \in k$, we have
\begin{align*}
	(\psi_\alpha^2 \circ f) \lp u \rp &=  (\iota_B \circ \pi \inv \circ \pi_{\beta} \circ f)  \lp u \rp = (\iota_B \circ \pi \inv \circ f_{\beta}) (u) = (\iota_B \circ g_\alpha)(u) = (\rho \circ X_{\alpha}) \lp u \sqrt{d} \rp
\end{align*}
Hence, for $v = (v_1, v_2) \in V_\alpha(k) = L = k \oplus k \sqrt{d}$, we have
\begin{align*}
	\Big( \psi_\alpha \circ V_\alpha(f) \Big) (v) &= \psi_\alpha \Big( f(v_1), f(v_2) \Big) = \Big( \psi_\alpha^1 \circ f(v_1) \Big) \cdot \Big( \psi_\alpha^2 \circ f(v_2) \Big) = \\
	&= \Big( \rho \circ X_\alpha(v_1) \Big) \cdot \Big( \rho \circ X_\alpha \lp v_1 \sqrt{d} \rp \Big) = \\
	&= \rho \circ X_\alpha \lp v_1 + v_2 \sqrt{d} \rp = (\rho \circ X_\alpha)(v)
\end{align*}
Thus, $\psi_\alpha \circ V_\alpha(f) = \rho \circ X_\alpha$, as needed.
\end{proof}

\noi To conclude this section, we now show the representation $\rho$ can be lifted to a representation $\wt \sig$ of the Steinberg group $\wt G(A)$. The precise statement is as follows.
% Note that since $f:k \to A$ has Zariski-dense image in $A$, for any root $\alpha$ the map $V_\alpha(f):V_\alpha(k) \to V_\alpha(A)$ has Zariski-dense image in $V_\alpha(A)$. If $\alpha$ is long this is vacuous, and if $\alpha$ is short this just says that applying $f$ to both components of $k^2$ has dense image in $A^2$.

\begin{proposition}
\label{existence of sigma tilde}
Let $\rho \colon G(k) \to \GL_m(K)$ be an abstract representation, with $K$ an algebraically closed field of characteristic 0. Set $H = \ov{\rho(G(k))}$, and let $A,f, \psi_\alpha$ be as in Proposition \ref{existence of A}. Let $\wt F:\wt G(k) \to \wt G(A)$ be the homomorphism induced by $f$ and $\pi_k \colon \tilde{G}(k) \to G(k)$ be the natural map. Then there exists a group homomorphism $\wt \sig:\wt G(A) \to H$ such that $\wt \sig \circ \wt F = \rho \circ \pi_k$ and $\wt \sig \circ \wt X_\alpha = \psi_\alpha$ for all $\alpha \in \Phi_k$.
\end{proposition}
\begin{proof}
In order for the relation $\wt \sig \circ \wt X_\alpha = \psi_\alpha$ to hold, we must define $\wt \sig$ on the generators of $\wt G(A)$ by
\[
	\wt \sig \Big( \wt X_\alpha(v) \Big) := \psi_\alpha(v).
\]
To show that $\wt \sig$ is well defined, we need to verify that the relations (R1) and (R2) hold, replacing $\wt X_\alpha$ with $\psi_\alpha$. For this, we imitate the proof of \cite[Proposition 4.2]{IR}, starting with (R1). To make the notation less burdensome, let us set $\tilde {f} = V_{\alpha}(f).$  Let $a,b \in \tilde{f}(V_\alpha(k))$, and choose $v,w \in V_\alpha(k)$ such that $\tilde{f}(v) = a$ and $\tilde{f}(w) = b$. %(by $f$ we really mean $V_\alpha(f)$).
Then
\begin{align*}
	\psi_\alpha(a) \cdot \psi_\alpha(b) &= \lp \psi_\alpha \circ \tilde{f}\rp (v) \cdot \lp \psi_\alpha \circ \tilde{f} \rp (w) = (\rho \circ X_\alpha)(v) \cdot (\rho \circ  X_\alpha)(w) = \rho \Big( X_\alpha(v) \cdot X_\alpha(w) \Big) = \\
	&= (\rho \circ X_\alpha)( v+w)  = \lp \psi_\alpha \circ \tilde{f} \rp (v+w)  = \psi_\alpha \lp \tilde{f}( v) + \tilde{f}( w) \rp = \psi_\alpha ( a+b).
\end{align*}
Thus, we have two regular maps $V_\alpha(A) \times V_\alpha(A) \to H$ given by
\[
	(a,b) \mapsto \psi_\alpha(a) \cdot \psi_\alpha(b) \ \ \ \text{and} \ \ \  (a,b) \mapsto \psi_\alpha(a+b)
\]
that agree on the Zariski-dense subset $\tilde{f}(V_\alpha(k)) \subset V_{\alpha}(A).$ So, they must coincide on all of $V_\alpha(A)$. This verifies (R1).
%Thus the $\psi_\alpha$ version of (R1) holds.
By a similar calculation, for any $\alpha,\beta \in \Phi_k$ with $\alpha \neq \pm \beta$, we have two regular maps $V_\alpha(A) \times V_\beta(A) \to H$ given by
\[
	(a,b) \mapsto \Big[ \psi_\alpha(a), \psi_\beta(b) \Big] \ \ \ \text{and} \ \ \  (a,b) \mapsto \prod_{i,j >0} \psi_{i\alpha+j\beta} \Big( N_{ij}^{\alpha \beta}(a,b) \Big)
\]
that agree on the Zariski-dense subset $\tilde{f}(V_\alpha(k))$, and hence on all of $V_\alpha(A)$. Thus, (R2) holds as well.
%Hence the $\psi_\alpha$ version of (R2) holds.
This shows that $\wt \sig$ is well defined, and, by construction, satisfies $\wt \sig \circ \wt X_\alpha = \psi_\alpha$ for all $\alpha \in \Phi_k$. Finally, for any $\alpha \in \Phi_k$ and $v \in V_\alpha(k)$, we have
\[
	(\wt \sig \circ \wt F) \Big( \wt X_\alpha(v) \Big) = \wt \sig \Big( \wt X_\alpha \lp \tilde{f}(v) \rp \Big) = \psi_\alpha \lp \tilde{f}(v) \rp = (\rho \circ X_\alpha)(v) = (\rho \circ \pi_k) \Big( \wt X_\alpha(v) \Big),
\]
from which it follows that $\wt \sig \circ \wt F = \rho \circ \pi_k$.
\end{proof}

\section{Rationality and conclusion of the proof}\label{S-Rationality}
\label{rationality}

We retain the notations of the previous section. Namely, we let $G = \SU_{2n}(L,h)$, consider an abstract representation $\rho \colon G(k) \to \GL_m(K)$, with $K$ an algebraically closed field of characteristic 0, and set $H =\ov{\rho(G(k))}$. Recall that by Proposition \ref{existence of A}, one can associate to $\rho$ an algebraic ring $A$ together with a ring homomorphism $f \colon k \to A$ with Zariski-dense image. Moreover, in Proposition \ref{existence of sigma tilde}, we constructed a group homomorphism $\wt \sig \colon \wt G(A) \to H$ that lifts $\rho$ to a representation of the Steinberg group $\wt G(A).$ More precisely, the diagrams formed by the solid arrows below commute
\begin{equation}
\label{diagram}
	\begin{tikzcd}
		\wt G(k) \arrow{r}{\wt F} \arrow{d}[swap]{\pi_k} & \wt G(A) \arrow[bend left]{ddr}{\wt \sig} \arrow{d}{\pi_A} \\
		G(k) \arrow{r}{F} \arrow[bend right]{drr}[swap]{\rho} & G(A) \arrow[dotted]{dr}{\sig} \\
		&& H
	\end{tikzcd}
\end{equation}
(here $F$ and $\wt F$ denote the group homomorphisms induced by $f$). To complete the proof of Theorem \ref{T-MainTheorem}, it remains to show the existence, under the assumptions of the theorem, of a morphism of algebraic groups $\sigma \colon G(A) \to H$ as indicated in the diagram.

In this analysis, we are implicitly using the functor of restriction of scalars to view $G(A)$ as an algebraic $K$-group. Namely, denoting by $G_A$ the base change of $G$ from $k$ to $A$, %and using the fact that $A$ is a finite-dimensional $K$-algebra,
the definition of the functor of restriction of scalars gives a natural isomorphism $r_{A/K} \colon G_A(A) \to R_{A/K} (G_A)(K)$. Since $A$ is a finite-dimensional $K$-algebra, $R_{A/K} (G_A)(K)$ is an affine algebraic $K$-group, and the isomorphism allows us to endow $G(A) = G_A(A)$ with the structure of an algebraic $K$-group.

Before addressing the algebraicity of $\sig$, we first note the following statement, which establishes some important properties of the group $H$.

%In this section we lift $\wt \sig$ to a group homomorphism $\ov \sig:G(A) \to \ov H$, then show that $\ov \sig$ is a morphism of algebraic groups. Then, assuming that $R_u(H)$ is commutative and $\characteristic K = 0$, we lift $\ov \sig$ to a morphism of algebraic groups $\sig:G(A) \to H$, obtaining a dotted arrow making the above commute. Restricting attention to the bottom triangle, we see that this $\sig$ is the claimed morphism of Theorem \ref{T-MainTheorem}.

%First, we show that $H$ is connected and perfect.

\begin{lemma}
\label{H is connected}
The homomorphism $\wt \sig:\wt G(A) \to H$ is surjective and the algebraic group $H$ is connected and perfect.
%$H = H^\circ = \wt \sig \lp \wt G(A) \rp = [H^\circ, H^\circ] = [H,H]$.
\end{lemma}
\begin{proof}
Let $\calH \subset H$ be the (abstract) subgroup generated by the elements $\psi_\alpha(v) =  \wt \sig \circ \wt X_\alpha (v)$ for all $\alpha \in \Phi_k$ and $v \in V_\alpha(A)$, where $\psi_{\alpha} \colon V_{\alpha}(A) \to H$ are the regular maps introduced in Proposition \ref{existence of A}. Since, by definition, the $\wt X_\alpha(v)$ generate $\wt G(A)$, their images generate $\wt \sig \lp \wt G(A) \rp$, so $\calH = \wt \sig(\wt G(A))$. Now, $A$ is connected by Remark \ref{properties of A}, so $V_{\alpha}(A)$ is connected, and hence
%$V_\alpha(A)$ is connected because $A$ is connected, and $\psi_\alpha$ is regular (hence continuous), so
$\psi_\alpha(V_\alpha(A))$ is connected. Thus, it follows from \cite[Proposition 2.2]{BorelAG} that
%Then by \cite[Proposition 2.2]{BorelAG}, it follows that
$\calH$ is Zariski-closed and connected, so $\calH \subset H^\circ$.
On the other hand, $\calH$ contains $\rho(E(k))$, which is Zariski-dense in $H$. So, $\calH$ is Zariski-dense in $H$, and since $\calH$ is closed, we see that $\calH = H$. This shows that $\calH = H = H^\circ$. Furthermore, by Lemma \ref{product of other generators}, $\wt G(A)$ is equal to its commutator subgroup, so the same is true for $\calH = \wt \sig(\wt G(A))$.
\end{proof}

In the remainder of the this section, we will complete the proof of Theorem \ref{T-MainTheorem} using a strategy inspired by that of \cite[\S\S 5,6]{IR}. Namely, let $Z(H)$ be the center of $H$, set $\ov H = H/ Z(H)$, and denote by $\nu \colon H \to \ov H$ the corresponding quotient map. We first show that $\wt \sigma$ gives rise to a group homomorphism $\ov \sig \colon G(A) \to \ov H$ satisfying $\ov \sigma \circ \pi_A = \nu \circ \wt \sigma$, and verify that $\ov \sig$ is in fact a morphism of algebraic groups. Then, using the assumption that the unipotent radical $U = R_u(H)$ is commutative (together with our standing hypothesis that ${\rm char}~K = 0$), we lift $\ov \sig$ to the required morphism of algebraic groups $\sig \colon G(A) \to H.$

\begin{proposition}
\label{existence of sigma bar}
There exists a group homomorphism $\ov \sig \colon G(A) \to \ov H$ such that $\ov \sig \circ \pi_A = \nu \circ \wt \sig$, where $\pi_A \colon \wt G(A) \to G(A)$ is the canonical map.
\end{proposition}
\begin{proof}
First, it follows from Remarks \ref{G=E} and \ref{properties of A} that $G(A) = E(A)$. Moreover, since $\pi_A$ is surjective, we have $E(A) \iso \wt G(A) / \ker \pi_A$. Now, according to Proposition \ref{central kernel}, $\ker \pi_A$ is central in $\wt G(A)$, and by Lemma \ref{H is connected}, $\wt \sig \colon \wt G(A) \to H$ is surjective. Consequently, we have $\wt \sig (\ker \pi_A) \subset Z(H)$, and hence
%so the restriction $\wt \sig |_{\ker \pi_A}$ maps into $Z(H)$. Hence
$\wt \sig$ induces a map $\ov \sig \colon G(A) \to \ov H$ on the quotients satisfying $\ov \sig \circ \pi_A = \nu \circ \wt \sig$.
\end{proof}

\noi Next, we establish:

\begin{proposition}
The group homomorphism $\ov \sig \colon G(A) \to \ov H$ from Proposition \ref{existence of sigma bar} is a morphism of algebraic groups.
\end{proposition}
\begin{proof}
By \cite[Proposition 2.2]{BorelAG}, %(or \cite[Lemma 2.14(iv)]{Stavrova}),
we can write $G(A) = E(A)$ as a product
\[
	%G(A) = %\prod_{\alpha \in \Phi_k} U_{\alpha}^{e_\alpha} =
	%\prod_{i=1}^m U_{\alpha_i}^{e_{\alpha_i}}
	G(A) = \prod_{i=1}^m U_{\alpha_i}^{e_i}
\]
for some sequence of roots $\{\alpha_1, \dots, \alpha_m \} \subset \Phi_k$, where $U_{\alpha_i} = X_{\alpha_i}(V_{\alpha_i}(A))$ is the root subgroup associated with $\alpha \in \Phi_k$ and each
%$e_{\alpha_i} = \pm 1.$
$e_i = \pm 1$. Let $X = \displaystyle \prod_{i=1}^m A_{\alpha_i}$ be a product of copies of $A$ indexed by the $\alpha_i$, and define a regular map $s \colon X \to G(A)$ by
\[
	s(a_1, \ldots, a_m) = X_{\alpha_1}(a_1)^{e_1} \cdots X_{\alpha_m}(a_m)^{e_m}.
\]
%The maps $X_{\alpha_i}$ are all regular, so $s$ is regular.
Let us also define a regular map $t' \colon X \to H$ by
\[
%	t':X \to H \qquad
	t'(a_1, \ldots, a_m) = \psi_{\alpha_1}(a_1)^{e_1} \cdots \psi_{\alpha_m}(a_m)^{e_m},
\]
where the $\psi_{\alpha_i}$ are the morphisms from Proposition \ref{existence of A}. %Since each $\psi_{\alpha_i}$ is regular, $t'$ is regular.
Set $t = \nu \circ t'$. One easily checks that
%A short calculation shows
$\ov \sig \circ s = t$. In particular, for any $x_1, x_2 \in X$, the condition $s(x_1) = s(x_2)$ implies $t(x_1) = t(x_2)$. So, by \cite[Lemma 3.10]{IR1}, $\ov \sig$ is a rational map. Hence, there is an open subset of $G(A)$ on which $\ov \sig$ is regular. Then, applying \cite[Lemma 3.12]{IR1}, we conclude that $\ov \sig$ is a morphism of algebraic groups.
\end{proof}

To conclude the argument, we will show that $\ov \sigma$ can be lifted to a morphism $\sigma \colon G(A) \to H.$ For this, we first discuss several preliminary statements, which are analogues in the present setting of results established in \cite[\S\S 5,6]{IR}.

Let $A$ be the algebraic ring associated with the representation $\rho$, and let $J \subset A$ be its Jacobson radical. As we already noted, $A$ is a finite-dimensional $K$-algebra; in particular, $A$ is artinian, and hence $J^d = \{0 \}$ for some $d \geq 1$ (see \cite[Proposition 8.4]{At}). Moreover, by the Wedderburn-Malcev Theorem (see \cite[Theorem 11.6]{Pierce}), there exists a semisimple subalgebra $\ov A \subset A$ such that $A = \ov A \oplus J$ as $K$-vector spaces and $\ov A \iso A/J$ as $K$-algebras. We note that \cite[Proposition 2.20]{IR} implies that $\ov A \simeq K \times \cdots \times K$ ($r$ copies). Since $G = \SU_{2n}(L,h)$ is $K$-isomorphic to $\mathrm{SL}_{2n}$, it follows that $G(\ov A)$ is a connected, simply connected, semisimple algebraic group. For the next statement, we consider the canonical homomorphism $A \to A/J$ and set
\[
	G(A, J) = \ker (G(A) \to G(A/J))
\]
to be the corresponding congruence subgroup. We then have the following.

\begin{lemma}
\label{levi decomposition}
\ \

\begin{enumerate}
	%\item $J$ is closed and nilpotent.
	%\item There exists a semisimple subalgebra $\ov A \subset A$ such that $A = \ov A \oplus J$ as $K$-vector spaces and $\ov A \iso A/J$ as $K$-algebras.
	%\item $G(\ov A)$ is a connected, simply connected, semisimple algebraic group.
	\item The congruence subgroup $G(A,J)$ is nilpotent.
	\item We have a Levi decomposition $G(A) = G(A,J) \semi G( \ov A)$.
\end{enumerate}
\end{lemma}
\begin{proof}
\noindent (1) Fixing an embedding $G(A) \into \GL_{2n}(A_L)$, it is straightforward to show that
\[
	\big[ G(A, J^a), G(A, J^b) \big] \subset G(A, J^{a+b})
\]
for any $a,b \in \Z_{\ge 1}$. Since $J$ is a nilpotent ideal, our claim follows.

\vskip2mm

\noindent (2) Using the fact that $G(A)$ is perfect (see Lemma \ref{product of other generators}), this statement is proved by the same argument as \cite[Proposition 6.5]{IR}.

\end{proof}

Next, we note the following analogue of \cite[Proposition 5.5]{IR} --- this result is proved exactly as in \cite{IR}, employing Lemma \ref{H is connected} in place of the corresponding statements in {\it loc. cit.} (let us point out that this is the first place where the assumption on the commutativity of the unipotent radical $U = R_u(H)$ is used).

%This will make use of our assumption that the unipotent radical $U = R_u(H)$ is commutative (which was actually not needed for the previous steps). We first note the following analogue of \cite[Proposition 5.5]{IR} --- this result is proved exactly as in \cite{IR}, employing Lemma \ref{H is connected} in place of the corresponding statements in {\it loc. cit.}.

%Until this point we have not assumed anything about the unipotent radical $R_u(H)$, or the characteristic of $K$. For further results, we need to assume that $R_u(H)$ is commutative and $\characteristic K = 0$ because of the following lemma.

\begin{lemma}
\label{property Z} Suppose $U = R_u(H)$ is commutative and $\characteristic K = 0.$ Then $Z(H) \cap U = \{e \}.$ Moreover, in this case, $Z(H)$ is finite and is contained in any Levi subgroup of $H$.
%\leavevmode
%\begin{enumerate}
%	\item If $U = R_u(H)$ is commutative and $\characteristic K = 0$, then $Z(H) \cap U = \lb e \rb$.
%	\item If $Z(H) \cap U = \lb e \rb$, then $Z(H)$ is finite.
%	\item If $Z(H) \cap U = \lb e \rb$ and $\characteristic K = 0$, then $Z(H)$ is contained in any Levi subgroup of $H$.
%\end{enumerate}
\end{lemma}

\noi We are now ready to complete the proof of Theorem \ref{T-MainTheorem} with the next result.

\begin{theorem}
Assume that $U = R_u(H)$ is commutative and $\characteristic K = 0$. Then there exists a morphism of algebraic groups $\sig \colon G(A) \to H$ making the diagram (\ref{diagram}) commute.
\end{theorem}
\begin{proof} (cf. \cite[Proposition 6.6]{IR})
Let $G(A) = G(A,J) \semi G(\ov A)$ be the Levi decomposition from Lemma \ref{levi decomposition}, and set
\[
	\ov U = \ov \sig (G(A,J)), \  \ \ \ov S = \ov \sig (G(\ov A)),  \ \ \ \text{and} \ \ \ S = (\nu \inv(\ov S))^{\circ},
\]
where $\nu \colon H \to \ov H$ is the quotient map. Then $\ov H = \ov U \semi \ov S$ and $H = U \semi S$ are also Levi decompositions. By Lemma \ref{property Z}, we have $Z(H) \subset S$. Consequently, $\ov S = S/Z(H)$ and the restriction $\nu|_U:U \to \ov U$ is an isomorphism.

Now, since the quotient map $\nu \colon H \to \ov H$ is a central isogeny, and, as we observed above, $G(\ov A)$ is simply connected, it follows from
%by Lemma \ref{levi decomposition}, $G(\ov A)$ is simply connected. Then by
\cite[Proposition 2.24(i)]{BT2} that there exists a morphism of algebraic groups $\sig_S \colon G(\ov A) \to S$ such that $\nu|_S \circ \sig_S = \ov \sig|_{G(\ov A)}$. Define $\sig_U = \nu|_U \inv \circ \lp \ov \sig|_{G(A,J)}\rp$. Then $$\sig = (\sig_U, \sig_S) \colon G(A) \to H$$ is a morphism of algebraic groups such that $\nu \circ \sig = \ov \sig$. It remains to show that $\sig$ makes the diagram (\ref{diagram}) commute. Define
\[
	\chi:\wt G(A) \to H, \qquad g \mapsto \wt \sig(g) \inv \cdot (\sig \circ \pi_A)(g).
\]
Since $\wt \sig$ and $\sig \circ \pi_A$ are group homomorphisms, so is $\chi$. Also, by Proposition \ref{existence of sigma bar}, we have $\ov \sig \circ \pi_A = \nu \circ \wt \sig$, which yields
%Putting $\ov \sig = \nu \circ \sig$ into $\ov \sig \circ \pi_A = \nu \circ \wt \sig$ (\ref{existence of sigma bar}), we obtain
$\nu \circ \sig \circ \pi_A = \nu \circ \wt \sig$. Thus, the image of $\chi$ is contained in $\ker \nu = Z(H)$. Since $\wt G(A)$ coincides with its commutator subgroup, we conclude that $\chi$ is trivial, and hence
%if the image is central in $H$ then $\chi$ must be trivial. Hence
$\sig \circ \pi_A = \wt \sig$. Finally, the equality
%fact that
$\sig \circ F = \rho$ follows from the commutativity of the rest of the diagram and the surjectivity of $\pi_k.$

%is a consequence of the fact that the rest of the diagram commutes and the fact that $\pi_k$ is surjective.
\end{proof}

%\noi As previously noted, this $\sig$ making diagram (\ref{diagram}) commute is the claimed morphism of Theorem \ref{T-MainTheorem}.

\vskip5mm

\noindent {\small {\bf Acknowledgements.} We would like to thank the anonymous referee for a careful reading of the paper and suggestions that helped to improve the exposition. The first author was partially supported by a Collaboration Grant for Mathematicians from the Simons Foundation.}

\vskip5mm

\bibliographystyle{amsplain}

\end{document}